\documentclass[11pt]{amsart}
\usepackage[utf8x]{inputenc}
\usepackage{amsthm,amssymb, amsmath}
\usepackage{graphicx}
\usepackage[all]{xy}
\usepackage{url}
\usepackage[dvipsnames]{xcolor}

\newtheorem{theorem}{Theorem}[section]
\newtheorem{corollary}[theorem]{Corollary}

\newtheorem{lemma}[theorem]{Lemma}
\newtheorem{proposition}[theorem]{Proposition}

\newtheorem{definition}[theorem]{Definition}
\newtheorem{remark}[theorem]{Remark}

\newcommand{\var}[2]{\textup{Var}_{#2}(#1)}

\title{Metrics on trees I. \\ The tower algorithm for interval maps.}
\author{Giulio Tiozzo}
\address{University of Toronto}
\email{tiozzo@math.utoronto.ca}

\date{\today}

\begin{document}
\maketitle

\begin{abstract}
We consider Milnor's \emph{tower algorithm} in the space of piecewise monotone maps, 
an iterative algorithm on the space of metrics which unifies, on the one hand, Thurston's iterative 
scheme which converges to holomorphic models, and, on the other hand, the theory of 
piecewise linear models coming from kneading theory. 
We prove that the algorithm converges for unimodal maps of high entropy, 
and provide examples where it does not converge for unimodal maps of low entropy and multimodal maps 
of higher degree. 
\end{abstract}

\section{Introduction}

A common theme in dynamics is the idea of ``geometrization": namely, one starts with a continuous map 
$f: X \to X$, seen as a topological dynamical system, and wants to produce a topologically conjugate (or semiconjugate) map to $f$ which is the ``nicest", meaning that it preserves an additional metric structure. 
For instance, according to a certain viewpoint, it is natural to ask for a representative which is holomorphic, while 
another natural choice would be to prefer a piecewise linear representative. 

In complex dynamics, the most striking example of this procedure is likely W. Thurston's theorem \cite{DH} on the realization of postcritically finite rational maps with given combinatorics. The theorem rests on an iterative scheme on Teichm\"uller space $\mathcal{T}$, which
is the space of marked complex structures on a (topological) surface.  The iteration is essentially given by pulling back the complex structure, and one shows that contracting properties of this map on $\mathcal{T}$ imply that an initial complex structure converges under iteration to a fixed point, which produces an invariant complex structure, in turn yielding a rational map homotopic to the initial one.  

In \cite{Mil}, Milnor proposes a new variation of Thurston's algorithm 
using a ``tower" of homeomorphisms. One of the advantages is that this algorithm can be 
at least defined even for postcritically infinite maps (even though its convergence remains to be proven). 
In \cite{Mil-slides}, he defines a generalized version of this algorithm, called the \emph{tower algorithm}. 
Such iteration starts with a topological dynamical system and aims to produce in the limit a
representative of the initial system which belongs to a special class we are interested in,
for instance holomorphic or piecewise linear.

In this paper, we consider the convergence properties of the tower algorithm  on the space of piecewise monotone interval maps. 
It turns out that it is convenient to frame the problem in terms of the space of \emph{metrics} on the interval, and consider iterations by a \emph{pullback operator}
on it. 

In Part II of this work, we shall generalize this approach to complex polynomials, showing its relation to core entropy and measured laminations.

\subsection{Critical values}

Let $I = [a,b]$ be a closed interval. A \emph{piecewise monotone map} $f$ is a continuous map $f : I \to I$ with $f(\partial I ) \subseteq \partial I$, for which there 
exist finitely many points 
$a = c_0 < c_1 < \dots < c_{d-1} < c_d = b$ in $I$, called \emph{critical points}, or \emph{turning points}, such that 
the restriction of $f$ to each $I_k := [c_{k-1}, c_{k}]$ is strictly monotone, and for any $1 \leq k \leq d-1$ the monotonicity of $f$ on $I_k$ and $I_{k+1}$ is different.  
The minimum such $d$ is called the \emph{degree} of the map $f$.
We denote as $\mathcal{F}_d$ the set of piecewise monotone maps of 
degree $d$.

\begin{definition}
We define the \emph{critical value vector} $CV(f)$ of the map $f$ as the vector 
$$CV(f) := (f(c_0), f(c_1), \dots, f(c_{d-1}), f(c_d)).$$
\end{definition}

A piecewise monotone map $f$ is \emph{piecewise linear of constant slope} if there exists $s$ such 
that the slope of $f$ equals $\pm s$ at each non-turning point. 

Following \cite{Mil-slides}, we give the definition:

\begin{definition}
A subset $\mathcal{G} \subseteq \mathcal{F}_d$ is \emph{parameterized by critical values} if for any $f \in \mathcal{F}_d$ 
there is one and only one $ g \in \mathcal{G}$ with the same critical value vector as $f$.
\end{definition}

The two main classes of maps parameterized by critical values we will consider are: 

\begin{enumerate}

\item The class $\mathcal{G}_{Pol}$ of polynomial maps with all critical points real and distinct, and in the interior of $I$; 

\item the class $\mathcal{G}_{CS}$ of piecewise linear monotone maps of constant slope.

\end{enumerate}

\subsection{An iterative procedure}

Any class of maps $\mathcal{G}$ parameterized by critical values defines an iterative  process in the space of piecewise monotone maps.
Let us fix such a class $\mathcal{G}$. The algorithm is as follows. 

\smallskip
\textbf{Step 1.} Given a piecewise monotone $f : I \to I$, let $g = g_f$ be the unique map in $\mathcal{G}$ such that 
$CV(f) = CV(g)$.

\smallskip
\textbf{Step 2.} Then, there exists a unique
homeomorphism $h = h_{f, g} : I \to I$ such that $f = g \circ h$.
The observation gives rise to the following diagram:
$$\xymatrix{
I \ar[dr]^g  \\
I \ar[r]^{f} \ar[u]^h & I  .
}$$

\smallskip
\textbf{Step 3.} Following (\cite{Mil-slides}, Slide 16), we define the operator $\Theta : \mathcal{F}_d \to \mathcal{F}_d$ as 
$$\Theta(f) := h_{f, g} \circ g_f.$$
The operator $\Theta(f)$ fits in the diagram:
$$\xymatrix{
I \ar[r]^{\Theta(f)} \ar[dr]^g & I \\
I \ar[r]^{f} \ar[u]^h & I \ar[u]_h .
}$$

\smallskip \textbf{Step 4.} By iterating this procedure, one produces a sequence $f_n := \Theta^n(f)$ of piecewise monotone maps, 
a sequence $(g_n)$ of maps in the class $\mathcal{G}$, and a sequence $(h_n)$ of homeomorphisms, 
which fit in the following tower:

$$\xymatrix{
I \ar[r]^{f_{n+1}} \ar[dr]^{g_n} & I  \\
I \ar[r]^{f_n} \ar[u]^{h_n} & I  \ar[u]_{h_n} \\
I \ar[r]^{f_2} \ar[dr]^{g_1} \ar@{-->}[u]  & I  \ar@{-->}[u] \\
I \ar[r]^{f_1} \ar[dr]^{g_0} \ar[u]^{h_1} & I \ar[u]_{h_1} \\
I \ar[r]^{f} \ar[u]^{h_0} & I \ar[u]_{h_0} 
}$$
where each $f_n$ is topologically conjugate to the initial $f = f_0$.

\begin{figure}
\begin{center}
\begin{minipage}{0.49 \textwidth} 
\includegraphics[width=1.0 \textwidth]{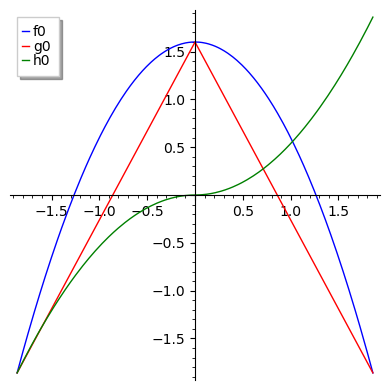}
\end{minipage}
\begin{minipage}{0.49 \textwidth}
\includegraphics[width=1.0 \textwidth]{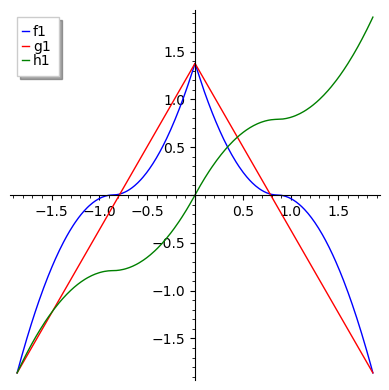}
\end{minipage}
\end{center}
\caption{The first two iterations of the tower algorithm. In most cases, $(f_n)$ converges 
to a piecewise linear map, and $(h_n)$ converges to the identity.}
\end{figure}

\subsection{Convergence} It is now natural to ask under which conditions this algorithm converges. 

\begin{enumerate}

\item The iteration of $\Theta_{Pol}$ provides an algorithm which, at least in the postcritically finite case, 
is equivalent to Thurston's algorithm; we present a more detailed description of the relationship in the Appendix. 

\item The iteration of $\Theta_{CS}$ is the main subject of the rest of the paper.

\end{enumerate}

\medskip

\noindent In the context of $\Theta_{CS}$, Milnor\footnote{For a collection of videos of examples and counterexamples to convergence, 
see \url{http://www.math.stonybrook.edu/~jack/BREMEN/}.} asks the following 
\smallskip

\noindent \textbf{Question }(\cite{Mil-slides}, Slide 18). 
For what maps $f = f_0$ does the sequence $(f_n)$ converge to a limit function $f_\infty$?

\medskip \noindent He also formulates the 
\smallskip 

\noindent \textbf{Conjecture }(\cite{Mil-slides}, Slide 19). For any (reasonable ?) $f_0$, the associated
sequence of constant slope maps $(g_n)$ converges, and yields
the ``correct" topological entropy 
$$h_{top}(f_0) = \log^+ (s(g_\infty))$$
where $s(g)$ denotes the slope. 

\medskip
\noindent In this paper, we show (Theorem \ref{T:unimodal-conv}) that the sequences $(f_n)$ and $(g_n)$ converge for unimodal maps with high entropy, and construct (Section \ref{S:counter}) counterexamples to convergence for unimodal maps with low entropy and for multimodal maps 
of degree $d \geq 3$. 

This problem turns out to be closely related to the classical problem of convergence of an initial distribution to the measure of maximal 
entropy; essentially, convergence of the algorithm is equivalent to mixing for the measure of maximal entropy. 
Thus, this is also connected to the spectral properties of the transfer operator, whose theory is by now highly developed (see \cite{Baladi-notes}; 
more details will be given in Section \ref{S:transfer}). 

\subsection*{Acknowledgements}
We thank John Milnor and Viviane Baladi for several useful discussions. 
The author is partially supported by NSERC and an Ontario Early Researcher Award. 

\section{Piecewise linear maps of constant slope}

We now consider the case of piecewise linear maps of constant slope.
The starting point of the iteration is given by the following observation. 

\begin{lemma}
Given a piecewise monotone $f : I \to I$, there exists a unique piecewise linear map $g$ of constant slope, such that $CV(f) = CV(g)$.
Moreover, there exists a unique homeomorphism $h : I \to I$ such that $f = g \circ h$.
\end{lemma}

\begin{proof}
Let us set $v_i := f(c_i)$ for $i = 0, \dots, d$. Let $\widetilde{c}_j$ be the turning points of $g$. Then, if the slope of $g$ is $s > 0$, we have 
$$ |v_{j+1} - v_j| =  s(\widetilde{c}_{j+1} - \widetilde{c}_j) \qquad \textup{for all }0 \leq j \leq d-1$$
hence, by summing over $j$, we obtain 
$$s = \frac{\sum_{j = 0}^{d-1} |v_{j+1} - v_j|}{\widetilde{c}_d - \widetilde{c}_0}.$$
Note that for simplicity we set $\widetilde{c}_0 = 0, \widetilde{c}_d = 1$. Moreover, $\widetilde{c}_i$ is given by 
$$\widetilde{c}_i - \widetilde{c}_0 = \frac{1}{s} \sum_{j = 0}^{i-1} |v_j - v_{j+1}|$$
and we define $g$ as 
$$g(x) := v_j + \epsilon s (x- \widetilde{c}_j) $$
for $\widetilde{c}_j \leq x \leq \widetilde{c}_{j+1}$ and $\epsilon = \pm 1$.
Further,  we note that $f : [c_j, c_{j+1}] \to [v_j, v_{j+1}]$ is a homeomorphism, and the same goes for $g : [\widetilde{c}_j, \widetilde{c}_{j+1}] \to [v_j, v_{j+1}]$.
Hence we define 
$$h\vert_{[c_j, c_{j+1}]} := g^{-1}\vert_{[v_j, v_{j+1}]} \circ f\vert_{[c_j, c_{j+1}]}.$$
\end{proof}

\subsection{The space of metrics}

An \emph{arc} is an unordered pair $[x, y]$ of distinct points of $I$. We think of $[x, y]$ as the 
subsegment of $I$ with endpoints $x, y$. We denote as $\mathcal{A}$ the set of arcs in $I$. 

We define  a \emph{metric} on $I$ to be a functional $m$ on the set $[x,y]$ of pairs of points in $I$, 
which we think of as assigning a length to each arc in the interval. 
We want the additional  properties that 
\begin{itemize}
\item[(i)] 
$m([x,y]) \geq 0$ for any $x,y \in I$;
\item[(ii)]
$m$ is additive on disjoint arcs: if $y$ lies between $x$ and $z$, then $m([x,z]) = m([x,y]) + m([y, z])$.
\end{itemize}

Let $\mathcal{M}(I)$ denote the space of metrics on $I$. This is a cone in a vector space, and has norm $\Vert m \Vert := m(I)$.
Clearly, any non-atomic Borel measure on the interval induces a metric. 
A metric $m$ has \emph{unit length} if $m(I) = 1$. We denote as $\mathcal{M}^1(I)$ the set of unit length metrics on $I$. 

\begin{remark}
On an interval, metrics and measures play almost the same role (except, of course, for the fact that for the moment we only require additivity 
rather than countable additivity); however, framing the problems in terms of metrics will lend itself more naturally to the generalization to trees and laminations
in the second part of this work. 
\end{remark}

\medskip
We define the \emph{weak topology} on the space of metrics by declaring that 
$m_n \rightharpoonup m_\infty$ if 
for any arc $J$ we have 
$$m_n(J) \to m_\infty(J).$$
Moreover, we define the \emph{strong topology} by setting that 
$m_n \to m_\infty$ if 
$$\sup_{J} |m_n(J) - m_\infty(J)| \to 0 \qquad \textup{as }n \to \infty.$$
Note that to any non-atomic, unit length metric of full support $m$ one can 
associate a homeomorphism $h : I \to I$ by defining 
$$h(x) = m([0,x]).$$
In fact, for any measure $m$, $h$ is weakly increasing and right-continuous: if $m$ has full support, 
then it is strictly increasing, while if $m$ is non-atomic, $h$ is continuous.

\subsection{The pullback operator}
Given $f$, we define a \emph{pullback operator} $f^\star : \mathcal{M}(I) \to \mathcal{M}(I)$ as follows.
For each interval $J \subseteq I$ and any $m \in \mathcal{M}(I)$, we let
$$(f^\star m)(J) := \sum_{k=1}^d m(f (J \cap I_k)).$$
By definition, the pullback operator is linear, and it is contravariant, i.e. $(f \circ g)^\star = g^\star \circ f^\star$.
Note that the kernel of $f^\star$ is the space of metrics $m$ such that $m(f(I))= 0$.

\medskip
\noindent \textbf{Example.} If $m_0$ is the metric induced by Lebesgue measure, then for any interval $J$ 
\begin{equation} \label{E:var}
(f^\star m_0)(J) = \textup{Var}_J (f)
\end{equation}
is the total variation of $f$ over $J$.

\begin{definition}
We call a metric $m$ \emph{linearly expanded} by $f$ if there exists $\lambda \in \mathbb{R}$ such that 
$$f^\star m = \lambda m.$$
\end{definition}

\noindent Linearly expanded, unit length metrics are fixed points of the operator
$$P(m) := \frac{f^\star m}{\Vert f^\star m \Vert}.$$

\begin{lemma} \label{L:semiconj}
A linearly expanded metric $m$ defines a semiconjugacy of $f$ to a piecewise linear map: indeed, if one defines 
$$h(x) := m([0, x])$$
then one has 
$$h \circ f = g \circ h$$
where $g$ is the piecewise linear map with slope $\lambda$ and 
critical points $h(c_i)$, where $c_i$ is a critical point for $f$.
\end{lemma}

\begin{proof}
Indeed, if $m$ is linearly expanded by $f$ and $x,y \in I_k$ for some $k$, we have 
\begin{align*}
f^\star m([x,y]) & = m([f(x), f(y)]) = | h(f(y))- h(f(x)) | \\
& = \lambda m([x,y]) = \lambda(h(y) - h(x)).
\end{align*}
\end{proof}

\begin{corollary}
The Lebesgue measure $m_0$ is linearly expanded by $f$ if and only if $f$ is a piecewise linear map with constant slope.
\end{corollary}

\begin{lemma}
There exists a metric $m$ on $I$ of unit length which is linearly expanded by $f$. 
\end{lemma}

\begin{proof}
Let us consider the space $\mathcal{M}^1(I)$ of unit length metrics on $I$. This is a convex subset of the topological vector space $\mathbb{R}^{\mathcal{A}}$, and it is continuous with respect to the weak topology. 
Then the operator 
$P :  \mathcal{M}^1(I) \to \mathcal{M}^1(I)$ defined as 
$$P(m) := \frac{f^\star m}{\Vert f^\star m \Vert}$$
is continuous, hence it has a fixed point. The fixed point is a metric which is linearly expanded by the dynamics.
\end{proof}

\subsection{Iteration in the space of metrics}

Let us denote as $m_0$ the unit length metric induced by normalized Lebesgue measure.
Moreover, we let $H_n := h_{n-1} \circ h_{n-2} \circ \dots \circ h_1 \circ h_0$, so that 
\begin{equation} \label{E:defHn}
f_n \circ H_n = H_n \circ f_0
\end{equation}
for any $n$. Denote as $s_i$ the slope of $g_i$. 

\begin{lemma}
We have for any $n \geq 1$ the identity
$$(f^\star)^{n}(m_0) = s_0 s_1 \dots s_{n-1} (H_n^\star)(m_0).$$ 
\end{lemma}

\begin{proof}
Let us first prove the case $n = 1$. We claim that if $g$ and $h$ are constructed from $f$ according to the 
above mentioned algorithm, then we have 
$$f^\star(m_0) = s \ h^\star(m_0)$$ 
where $s$ is the slope of $g$. 
From $f = g \circ h$ we get $f^\star = h^\star \circ g^\star$.
Then if $m_0$ is the Lebesgue measure, we have, since $g$ is piecewise linear with slope $s$
$$f^\star(m_0) = h^\star (g^\star m_0) = h^\star (s m _0) = s h^\star (m_0).$$
To prove the general claim, let us proceed by induction. Note that the previous argument with $f= f_n$, $g = g_n$ and 
$h = h_n$ yields
\begin{equation} \label{E:pullback}
f_n^\star(m_0) = s_n \ h_n^\star(m_0)
\end{equation}
where $s_n$ is the slope of $g_n$. Now, the inductive hypothesis is that 
$$(f^\star)^n(m_0) = s_0 \dots s_{n-1} H_{n}^\star(m_0).$$
By applying $f^\star$ to both sides, one gets
$$(f^\star)^{n+1}(m_0) = s_0 \dots s_{n-1} f^\star H_{n}^\star(m_0) $$
and using $f_n \circ H_{n} = H_{n} \circ f$ and equation \eqref{E:pullback} we get 
$$ (f^\star)^{n+1}(m_0) = s_0 \dots s_{n-1} H_{n}^\star f_n^\star(m_0) = s_0 \dots s_{n-1} s_n H_{n}^\star h_n^\star(m_0)$$
which yields the claim as $H_{n}^\star \circ h_n^\star = H_{n+1}^\star$ by construction.
\end{proof}

\begin{corollary} 
We have for any $n \geq 0$ the identity
$$P^{n}(m_0) = (H_n^\star)(m_0).$$
\end{corollary}

Hence we can look at the iterative procedure described above in the following way: 
we start with 
$$m_{n} := \frac{ (f^\star)^n(m_0) }{\Vert (f^\star)^n(m_0) \Vert},$$
then there exists a unique homeomorphism $H_{n}$ such that 
$$m_{n} = H_{n}^\star(m_0).$$
Finally, we get 
$$f_{n} = H_{n} \circ f \circ (H_{n})^{-1}.$$

\subsection{Hofbauer towers} 

Let $C$ be the set of turning points of $f$, and let $I_1, \dots, I_d$ be the closures of the connected components of $I \setminus C$. 
Let $\mathcal{D}$ be the smallest set of intervals in $I$ with the two properties: 
\begin{enumerate}
\item $I_1, I_2, \dots, I_d$ belong to $\mathcal{D}$; 

\item if $J$ belongs to $\mathcal{D}$, then $f(J) \cap I_i$ belongs to $\mathcal{D}$ 
for any $i$ such that the interiors of $f(J) $ and $I_i$ intersect. 

\end{enumerate}
Then, if we consider $\ell^1(\mathcal{D}) := \{ \sum_{J \in \mathcal{D}} a_J e_J \ : \ \sum |a_J| < \infty \}$,  we have the map 
$$p : \mathcal{M}(I) \to \ell^1(\mathcal{D})$$
defined as 
$$p(m) := \sum_{J \in \mathcal{D}} m(J) e_J.$$

\begin{lemma}
The map $p$ semiconjugates the pullback map $f^\star$ on $\mathcal{M}(I)$ to a linear map 
$T: \ell^1(\mathcal{D}) \to \ell^1(\mathcal{D})$
so that $p \circ f^\star = T \circ p$.
\end{lemma}

\begin{proof}
We define the map $T$ on each basis vector $e_J$ as 
$$T(e_J) := \sum_{i=1}^d e_{f(J) \cap I_i}$$
and extend it by linearity.
\end{proof}

The above space $\ell^1(\mathcal{D})$ has been defined by Hofbauer, 
who shows (\cite[Theorem 4]{Ho}) that 
$$h_{top}(f) = \log \rho(T)$$
where $T$ is the adjacency operator of the graph in $\ell^1(\mathcal{D})$ and $\rho$ is the spectral radius. 
He also proves that every topologically transitive piecewise monotone map has a unique measure of maximal entropy.
See also Raith \cite{Ra}.

\section{Transfer operators} \label{S:transfer}

Given a partition  $a = x_0 < x_1 < \dots < x_{k-1} < x_k = b$ of $[a, b]$, denoted by $\mathcal{P}$, we define 
$\textup{Var}(\varphi, \mathcal{P}) :=  \sum_{i = 0}^{k-1} |\varphi(x_{i+1}) - \varphi(x_i)|$.
A function $\varphi : I \to \mathbb{R}$ has \emph{bounded variation} if 
$$\textup{Var}_I (\varphi) := \sup_{\mathcal{P}} \textup{Var}(\varphi, \mathcal{P}) < \infty.$$ 
Given a measure $\mu$ on $I$, we consider the space $BV$ of functions of bounded variation, 
with norm 
$$\Vert \varphi \Vert_{BV} := \textup{Var}_I (\varphi) + \Vert \varphi \Vert_\infty.$$
The space $\mathcal{S}$ of \emph{step functions} is the space of all finite linear combinations 
of the characteristic functions $\chi_J$ of all subintervals $J \subset I$. 
A metric $m$ defines a linear operator $\mathbb{E}_m : \mathcal{S} \to \mathbb{R}$ by
$$\mathbb{E}_m(\varphi) := \sum_J a_J m(J)$$
if $\varphi = \sum_J a_J \chi_J$. We also write $\int \varphi \ dm$ instead of $\mathbb{E}_m(\varphi)$. 
Recall the well-known

\begin{lemma}  \label{L:dense}
The space $\mathcal{S}$ is dense in the space BV, with respect to the $\sup$ norm.
\end{lemma}

In fact, the uniform closure of the set of step functions is the space of \emph{regulated functions}, i.e. functions 
which admit both left and right limits everywhere, and functions of bounded variation are regulated.

Moreover, following \cite{Rugh}, we define the space $SBV(I)$ of \emph{step bounded variation functions} as the closure of the space of step functions \emph{with respect to the BV norm}. This is a smaller space than $BV(I)$. 

By Lemma \ref{L:dense}, we can extend the operator $\mathbb{E}_m$ to $\mathbb{E}_m : BV \to \mathbb{R}$, hence obtaining a
pairing $\langle \cdot, \cdot \rangle : BV \times \mathcal{M}(I) \to \mathbb{R}$,
$$\langle \varphi, m \rangle := \int \varphi \ dm.$$

\subsection{The transfer operator}

Let us now recall the spectral theory of transfer operators on the space BV. For further references, see e.g. \cite{Baladi-book}.

Let $\mathcal{L} : \textup{BV} \to \textup{BV}$ be the transfer operator, with potential $g \equiv 1$. That is,
$$(\mathcal{L} \varphi)(x) := \sum_{i = 1}^d \varphi (f_i^{-1} x) \chi_{f(I_i)}(x) $$
where $f_i^{-1} : f(I_i) \to I_i$ are the local inverses\footnote{Note that different sources (e.g. \cite{Baladi-Keller}, \cite{Rugh}) have slightly different definitions of $\mathcal{L}$, depending on considering the endpoints 
of $I_i$ or not. However, since the definitions only differ on finitely many points, such modifications induce the same operator on $BV/\mathcal{N}$, hence they turn out to have the same discrete spectrum by \cite[Proposition 3.4]{Baladi-book}.} of $f$. 

\begin{lemma}
For any interval $I$ and any metric $m$, we have 
$$\langle \mathcal{L} \chi_I, m \rangle = \langle \chi_I, f^\star m \rangle.$$
\end{lemma}

\begin{proof}
By definition we have 
\begin{align*}
\int \mathcal{L} \chi_I \ dm & = \sum_{i = 1}^d \int (\chi_I \circ f_i^{-1}) \chi_{f(I_i)} \ dm\\
& = \sum_{i = 1}^d \int \chi_{f(I \cap I_i)} \ dm \\
& = \sum_{i = 1}^d m(f(I \cap I_i)) \\
& = f^\star m(I).
\end{align*}
\end{proof}

As a corollary, using Lemma \ref{L:dense}, 
we obtain the duality relation 
$$\langle \mathcal{L} \varphi, m \rangle = \langle \varphi, f^\star m \rangle$$
for any $\varphi \in BV$, $m \in \mathcal{M}$. 
Moreover, by setting $m = m_0$ in the previous proof, we obtain
$$\textup{Var}_J (f^n) = \Vert \mathcal{L}^n \chi_J \Vert_{L^1}$$
where we consider the $L^1$ space with respect to Lebesgue measure. 

\subsection{Spectral decomposition}

Let $f : I \to I$ be a piecewise monotone map, and let $\lambda := e^{h_{top}(f)}$. 

We denote as $\ell(f, J)$ the number of \emph{laps}, i.e. the number of intervals of monotonicity, of the restriction of $f$ to $J$; 
we let $\ell(f) := \ell(f, I)$ the total number of laps. 

By \cite[Chapter III.5]{HK2}, there exists a measure $\mu$ on $I$ 
such that 
$$\int_I \mathcal{L}(\varphi) \ d \mu =  \lambda \int_I \varphi \ d \mu$$
for all bounded, measurable $\varphi$. Note that in general $\mu$ need not have full support. 

Following \cite{Baladi-Keller}, we have a spectral decomposition. 

\begin{theorem} \label{T:sp-dec-II}
Suppose $f : I \to I$ is a piecewise monotone map with $h_{top}(f) > 0$.
Let $\mathcal{L} : BV \to BV$ be the transfer operator. Then: 
\begin{enumerate}
\item the spectral radius of $\mathcal{L}$ equals $\lambda = e^{h_{top}(f)}$; 
\item there exists a spectral decomposition 
$$\mathcal{L} = \sum_{i = 1}^r \lambda_i P_i ( P_i + N_i) + P_{res} \mathcal{L}$$
where $\lambda = \lambda_1, \lambda_2, \dots, \lambda_r$ are the eigenvalues of modulus $\lambda$, and each $P_i$ is a spectral projector of finite rank.
Moreover, each $N_i$ is nilpotent, and $P_i N_i = N_i P_i = N_i$.
Further, there exists $\mu < \lambda$ and $c > 0$ such that 
$$\Vert P_{res} \mathcal{L}^n \Vert_{BV} \leq c \mu^n \qquad \textup{for any }n \geq 1.$$ 
\item
If $\varphi$ is an eigenvector with eigenvalue $\lambda$, then it is non-negative, and the measure $\nu$ defined by $d \nu = \varphi \ d \mu$ is invariant under $f$. 
\end{enumerate}
\end{theorem}

\begin{proof}
Let $\mathcal{N}$ be the space of functions of bounded variation that vanish except on a countable set. 
By \cite[Proposition 3.4]{Baladi-book}, the spectrum on $BV$ and on $BV/\mathcal{N}$ coincide outside the disk of radius $1$. 
To prove (1), note that, by \cite[Theorem 3.2 (c)]{Baladi-book}, the spectral radius of $\mathcal{L}$ on $BV/\mathcal{N}$ is 
$$R := \lim_{n \to \infty} (\Vert \mathcal{L}^n 1 \Vert_\infty)^{1/n}$$
and let us note that 
$$\var{f^n}{I} = \langle \mathcal{L}^n 1, m_0 \rangle \leq \Vert \mathcal{L}^n 1 \Vert_\infty \leq \ell(f^n)$$
and by Misiurewicz-Szlenk \cite{MS} and the definition of topological entropy 
$$\lim_{n \to \infty} (\var{f^n}{I})^{1/n} = e^{h_{top}(f)} = \lim_{n \to \infty} (\ell(f^n))^{1/n}$$
hence 
$$R = e^{h_{top}(f)},$$
see also \cite[Theorem 3.3]{Baladi-book}.
Now, (2) is \cite[Theorem 1]{Baladi-Keller} and follows by quasicompactness of $\mathcal{L}$. 
Finally, (3) is \cite[Theorem 3.2 (c)]{Baladi-book}. 
\end{proof}

Similar decompositions are given in \cite{Ry}, \cite{HK} (but under the assumption that $\mu$ has full support, so there is no nilpotent part). 

For piecewise monotone maps, the peripheral eigenvalues of the transfer operator acting on $BV$ are in correspondence 
with the zeros of the kneading determinant. 

\begin{theorem} \label{T:eigen-knead}
Let $f$ be a piecewise monotone map with $h_{top}(f) > 0$. Then a complex number $t$ with 
$0 < |t| < 1$  
is a zero of the Milnor-Thurston determinant $D_f(t)$ if and only if $t^{-1}$ is an eigenvalue of the 
transfer operator $\mathcal{L}$ acting on the space $BV$. 
Moreover, the algebraic multiplicities are the same.
\end{theorem}

\begin{proof}
The statement follows from work of Baladi-Ruelle \cite{Baladi-Ruelle}, who also deal with non-constant weights (see also \cite{Baladi-notes}). 
In the constant weight case we are interested in, a direct way is to invoke \cite[Theorem 3.1]{Rugh}, whereby the zeros of the Milnor-Thurston determinant correspond to the eigenvalues of modulus $> 1$ of the transfer operator on the space $SBV(a,b)$ of step bounded variation functions; moreover, by \cite[Theorem 3.3]{Rugh}, these eigenvalues also correspond to the eigenvalues of the transfer operator acting on $BV(a, b)$. 
\end{proof}

\begin{remark}
If the partition $\bigsqcup_{i = 0}^{d-1} [c_i, c_{i+1})$ is generating, then by \cite[Theorem 2]{Baladi-Keller} the eigenvalues of $\mathcal{L}$ 
coincide with the poles of the zeta function $\zeta(t)$. Moreover, by \cite[Theorem 9.2]{MT}
$$D_f(t) = \frac{\chi(t)}{\widehat{\zeta}(t)},$$ 
where $\chi(t)$ is a cyclotomic polynomial and $\widehat{\zeta}(t)$ is the \emph{reduced} zeta function. 
Since there are no homtervals, the reduced zeta function and the usual zeta function $\zeta(t)$ are the same, possibly up to 
cyclotomic factors.
\end{remark}

\subsection{The topologically mixing case}

Recall a map $f : I \to I$ is \emph{topologically mixing} if for any nonempty, open sets $U, V$ there exists $n_0$ such that $f^{-n}(U) \cap V \neq \emptyset$
for any $n \geq n_0$.
In this case we have: 
 
\begin{proposition}[\cite{Baladi-book}, Proposition 3.6] \label{P:top-trans}
If $f$ is topologically transitive, then $\lambda = e^{h_{top}(f)}$ is a simple eigenvalue of $\mathcal{L}$. 
Moreover, if $f$ is topologically mixing, then $\mathcal{L}$ does not have any other eigenvalue of modulus $\lambda$.
\end{proposition}

Recall that the \emph{postcritical set} of $f$ is  $P_f := \overline{\bigcup_{i = 1}^{d-1} \bigcup_{n \geq 1} f^n(c_i)}$, and 
the \emph{core} of $f$, denoted $K_f$, is the convex hull of $P_f$. Note that $f(K_f) \subseteq K_f$.

\begin{lemma} \label{L:nocore}
Any eigenvector of $\mathcal{L}$ of eigenvalue $\xi$ with $|\xi| > 1$ vanishes on $I \setminus K_f$. 
\end{lemma}

\begin{proof}
Note that $I \setminus K_f$ has two connected components, let us call them $I^+$ and $I^-$; moreover, $f$ is injective on each of them, and  
$f^{-1}(I^+) \subseteq I^+$ and $f^{-1}(I^-) \subseteq I^+$. 
Thus, for any $x \in I \setminus K_f$ and for any $n \geq 1$, the set $f^{-n}(x)$ contains at most two elements. 
Hence, if $\varphi$ is an eigenvector of $\mathcal{L}$ of eigenvalue $\xi$, we have for any $x \in I \setminus K_f$, 
$$|\xi|^n |\varphi(x)| = |\mathcal{L}^n \varphi(x)| \leq 2 \sup_I |\varphi|$$
for any $n$, so $\varphi(x) = 0$. 
\end{proof}

We say $f$ is \emph{topologically mixing on the core} if the restriction $f \vert_{K_f}$ is topologically mixing. 
Moreover, $f$ is \emph{locally eventually onto the core} if for any closed interval $J \subseteq K_f$ with nonempty interior
there exists $n > 0$ such that $f^n(J) = K_f$. This implies topologically mixing on the core.

As a consequence of Proposition \ref{P:top-trans} and Lemma \ref{L:nocore}, we obtain the following (well-known) spectral decomposition: 

\begin{theorem} \label{T:mixing-spectrum}
Suppose $f : I \to I$ is a piecewise monotone map with $h_{top}(f) > 0$ and topologically mixing on the core. 
Let $\mathcal{L} : BV \to BV$ be the transfer operator. Then: 
\begin{enumerate}
\item the spectral radius of $\mathcal{L}$ equals $\lambda = e^{h_{top}(f)}$, and 
there is a unique eigenvector of eigenvalue $\lambda$; 
\item the rest of the spectrum is contained in $\{ z \ : \ |z| \leq \rho \}$ for some $\rho < \lambda$.
\end{enumerate}
\end{theorem}

Eigenvectors of $\mathcal{L}$ are related to maximal measures for $f$. Consider the map $i : \ell^1(\mathcal{D}) \to BV$ given by 
$$i(u) := \sum_{J \in \mathcal{D}} u_J \chi_J.$$
Then one has the semiconjugacy $\mathcal{L}(i(u)) = i(T(u))$. Moreover:

\begin{theorem}[\cite{HK}, Theorem 3]
A complex number $\lambda$ with $|\lambda| > \rho_{ess}$ is an eigenvalue of the transfer operator on $BV$ if and only if it is an eigenvalue of the adjacency operator on $\ell^1(\mathcal{D})$. 
\end{theorem}

\subsection{Unimodal maps}
Let us now assume that $f$ is a unimodal map, i.e. $d = 2$. We shall normalize $f$ so that $I = [0, 1]$ and $f(0) = f(1) = 0$. 
We denote as $c$ the (unique) critical point of $f$.  The \emph{core} of a unimodal map $f$ is the interval $K_f := [f^2(c), f(c)]$.

In the unimodal case, the following dichotomy is very useful. It is inspired by \cite[Proposition 2.5.5]{vS}.

\begin{definition}
We call an interval $J$ \emph{fast} if there exists $k \geq 0$ such that both $f^{k}(J)$ and $f^{k+1}(J)$ contain $c$.  
Otherwise, we call $J$ a \emph{slow interval}. 
\end{definition}

Clearly, if $J$ is slow, then $f(J)$ is also slow; similarly, if $J$ is fast and $f(J') = J$, then $J'$ is also fast. 

\begin{lemma}
Let $f : I \to I$ be a unimodal map, and let $J \subseteq I$ be a closed interval. Then: 
\begin{enumerate}
\item If $J$ is fast, there exists $n$ such that $f^n(J) = [f(c), f^2(c)]$; 
\item If $J$ is slow, we have 
$$\textup{Var}_J f^n \leq \ell(f^n, J) \leq 2^{n/2} \qquad \textup{for any }n \geq 1.$$ 
\end{enumerate}
\end{lemma}

\begin{proof}
An interval $J$ is fast if there exists $k \geq 0$ such that $f^k(J)$ and $f^{k+1}(J)$ both contain the turning point. 
In that case, $f^{k+2}(J)$ contains $[f(c), f^2(c)]$, which is the core of the interval. 
If $J$ is slow, then for any $n$ the degree of $f^2$ restricted to $f^n(J)$ is at most $2$, hence 
$$\ell(f^{n+2}, J) \leq \ell(f^2, f^n(J)) \ell(f^n, J) \leq 2 \ell(f^n, J),$$
which by induction yields the claim.
\end{proof}

Note that $J$ is fast if and only if 
$$\liminf_{n \to \infty} \frac{\var{f^n}{J}}{\lambda^n} > 0.$$
Hence if $J_1, J_2$ are slow intervals with non-empty intersection, the union $J_1 \cup J_2$ is also slow. 

\bigskip

For unimodal maps with high entropy, the spectral decomposition has a simpler form.

\begin{theorem} \label{T:unispectrum}
Suppose $f : I \to I$ is a unimodal map with $h_{top}(f) > \frac{1}{2} \log 2$.
Let $\mathcal{L} : BV \to BV$ be the transfer operator. Then: 
\begin{enumerate}
\item the spectral radius of $\mathcal{L}$ equals $\lambda = e^{h_{top}(f)}$, and 
there is a unique eigenvector of eigenvalue $\lambda$; 
\item the rest of the spectrum is contained in $\{ z \ : \ |z| \leq \rho \}$ for some $\rho < \lambda$.
\end{enumerate}
\end{theorem}

We shall prove the spectral properties of unimodal maps by relating them to the zeros of the kneading determinant. 

\subsection*{Kneading theory}
Let us recall the definition of kneading determinant from \cite{MT}. 
In order to capture the symbolic dynamics of 
$f$, one defines the \emph{address} of a point $x \neq c$ as 
$$A(x)  := \left\{ \begin{array}{ll} +1 & \textup{if }  x \in [0, c) \\ 
-1 & \textup{if } x \in (c, 1]. 
\end{array} \right.$$
The reason for this definition is to capture the monotonicity of $x$, i.e. $A(x) = +1$ iff $f$ is increasing in a neighbourhood of $x$. 
The \emph{kneading sequence} of $f$ is then defined as the sequence of addresses of the iterates of the turning point; namely, 
for any $k \geq 1$ set, if $f^k(c) \neq c$, 
$$\epsilon_k := A(f^k(c))$$
while, if $f^k(c) = c$, then set 
$$\epsilon_k := \lim_{x \to c} A(f^k(x))$$
which is still well-defined as $f$ ``folds" a neighbourhood of $c$. Finally, one defines 
$\eta_k := \prod_{j = 1}^k \epsilon_j$.
Then the \emph{kneading series} associated to $f$ is the power series 
$$D_f(t) := 1 + \sum_{k = 1}^\infty  \eta_{k} t^k.$$
Note that the coefficients of $D_f(t)$ are uniformly bounded, hence the power series 
defines a holomorphic function in the unit disk $\{ t \in \mathbb{C} \ : \ |t| < 1\}$. 

\begin{proposition} \label{P:no-periph}
Let $f$ be a unimodal map with $h_{top}(f) = \log \lambda> \frac{1}{2} \log 2$. Then $\lambda^{-1}$ is a simple zero of the kneading 
determinant $D_f(t)$,
and there are no other zeros on the closed disk $\{ z \ : \ |z| \leq \lambda^{-1} \}$. 
\end{proposition}

First of all, for unimodal maps of positive entropy, the smallest real zero of $D_f(t)$ simple: 

\begin{lemma}[\cite{Ti}, Theorem 3.1] \label{T:simple}
Let $f : I \to I$ be a unimodal map with topological entropy $h_{top}(f) > 0$.  Denote as $D_f(t)$ its kneading determinant, and let $s = e^{h_{top}(f)}$. Then $r = \frac{1}{s}$ is a simple root of $D_f(t)$.
\end{lemma}

This fact is closely related to the fact (see Raith \cite[Theorem 5]{Ra}) that any unimodal map of positive entropy has a unique measure of maximal entropy. 

\begin{proof}[Proof of Proposition \ref{P:no-periph}]
By Milnor-Thurston \cite{MT}, for any unimodal map $f$ of entropy $h_{top}(f) = \log \lambda$ there exists a semiconjugacy $\pi : I \to J$ of $f$ to a piecewise linear unimodal map $g : J \to J$ of slope $\pm \lambda$. That is, there is a continuous, surjective, weakly monotone map $\pi : I \to J$ such that $\pi \circ f = g \circ \pi$.

If $g$ is piecewise linear with slope $\lambda > \sqrt{2}$, then it is locally eventually onto the core \cite[Theorem 2]{Bo}. 
Indeed, if $J$ is a slow interval, $g^2\vert_{J}$ is piecewise monotone of degree $2$ and slope $\pm \lambda^2$, hence its length satisfies
$$m_0(g^{2}(J)) \geq \frac{\lambda^2}{2} m_0(J)$$
which implies for any $n \geq 1$
$$m_0(g^{2n}(J)) \geq \left(\frac{\lambda^2}{2}\right)^n m_0(J)$$
which tends to $\infty$ as $n \to \infty$, contradiction. Hence, there are no slow intervals, 
so $g$ is locally eventually onto the core.

Hence, its kneading 
determinant $D_g(t)$ 
has only one simple zero $t = e^{-h_{top}(g)} = e^{-h_{top}(f)}$, and no other zero of modulus $\leq \lambda^{-1}$.

Now, denote as $c$ the turning point of $f$, and let $\widetilde{c} = \pi(c)$ be the turning point of $g$. Moreover, let $L := \pi^{-1}(\widetilde{c})$, which is a closed interval containing $c$. There are two cases: 
\begin{enumerate}
\item
either $f^n(c) \notin L$ for all $n \geq 1$. This implies that 
$$D_f(t) = D_g(t)$$
hence the claim follows, since from above we know that $D_g(t)$ does not have any other root of modulus $\lambda^{-1}$.
\item 
Otherwise, there exists $n$ such that $f^n(c) \in L$. Let $p$ be the smallest such $n$. This implies that $g^p(\widetilde{c}) = g^p(\pi(c)) = \pi(f^p(c)) = \widetilde{c}$, since $f^p(c) \in L$. Then we get the factorisation
$$D_f(t) = \widetilde{D}_g(t) D_h(t^p)$$
where $h = f^p\mid_L$ is the first return map of $f$ to $L$, which is also a unimodal map, and $\widetilde{D}_g(t)$ is the polynomial (of degree $p-1$)
such that 
$$D_g(t) = \frac{\widetilde{D}_g(t)}{1-t^p}.$$
Since $h$ is a unimodal map, every root of $D_h$ is of modulus $\geq \frac{1}{2}$, hence, since $p \geq 2$, we have 
that every root of $D_h(t^p)$ has modulus at least $\sqrt{2}$. 
\end{enumerate}
This completes the proof. 
\end{proof}

\begin{proof}[Proof of Theorem \ref{T:unispectrum}]
It follows from Theorem \ref{T:eigen-knead} and Proposition \ref{P:no-periph}. 
\end{proof}

Note that we have proved the following fact, which may be well-known but we could not find in the literature:

\begin{corollary} \label{C:mixing}
A unimodal map $f$ with $h_{top}(f) > \frac{1}{2} \log(2)$ is mixing for the measure of maximal entropy\footnote{Let us remark that the measure of maximal entropy is in general \emph{not} absolutely continuous w.r.t. Lebesgue, 
and establishing mixing properties for the absolutely continuous invariant measure is a much more delicate problem (see e.g. \cite{Blokh-Lyubich}, \cite{Avila-Moreira}).}.
\end{corollary}

\section{Convergence in the unimodal case}

Let us now state our main convergence result. 

\begin{theorem} \label{T:unimodal-conv}
Let $f : I \to I$ be a unimodal map with $h_{top}(f) > \frac{1}{2} \log 2$.
Then:
\begin{enumerate}
\item for any $\varphi$ in $BV$, the sequence $(\mathcal{L}^n \varphi)$ converges to an eigenvector of $\mathcal{L}$
of eigenvalue $\lambda = e^{h_{top}(f)}$;
\item
the sequence $(H_n)$ of homeomorphisms converges to a limit map $H_\infty$.
The limit $H_\infty$ semiconjugates the dynamics $f$ to the piecewise linear map of the same entropy;
\item
the sequence $(f_n)$ converges uniformly to a unimodal map $f_\infty$
of constant slope 
$$s = e^{h_{top}(f)};$$ 
\item moreover, the sequences $(g_n)$ and $(h_n)$ also converge uniformly 
to limits $g_\infty, h_\infty$.
\end{enumerate}
\end{theorem}

Note that the example of anomalous convergence from \cite[Slide 21]{Mil-slides} 
with $f_0(x) = 2.8 x (1-x)$ has zero entropy, hence it does not contradict Theorem \ref{T:unimodal-conv}. 
More counterexamples will be discussed in the next Section.

\begin{proof}
Consider 
$$m_n := H_n^\star (m_0) = \frac{(f^\star)^n(m_0)}{\Vert (f^\star)^n(m_0) \Vert}.$$
Then, if we let $M_n := \Vert (f^\star)^n(m_0) \Vert$ and $J = [x, y]$, we have by Theorem \ref{T:unispectrum}
$$\frac{\mathcal{L}^n (\chi_J)}{M_n} \to c \cdot P(\chi_J)$$
where $P = P_0$ is the projector to the leading eigenspace of $\mathcal{L}$, and $c := \frac{1}{\langle P(1), m_0 \rangle}$.
Thus
\begin{align} \label{E:Hn}
H_n(y) - H_n(x) & = H_n^\star(m_0)([x, y]) \\
& = m_n(J) \\
& = \langle m_n, \chi_J \rangle \\
& = \frac{\langle (f^\star)^n(m_0), \chi_J \rangle}{M_n} \\
\label{E:Hn1}
& = \frac{\langle m_0, \mathcal{L}^n (\chi_J) \rangle}{M_n} \to c \cdot \langle m_0, P(\chi_J) \rangle
\end{align}
converges as $n \to \infty$. The convergence is uniform in $x, y$ as
\begin{align*} 
\lambda^{-n} \langle (P_{res})^n(\chi_{[x, y]}), m_0 \rangle &  \leq \lambda^{-n} \Vert (P_{res})^n \Vert_{BV} \Vert \chi_{[x, y]} \Vert_{BV} \\
&  \leq  3 \lambda^{-n} \Vert (P_{res})^n \Vert_{BV} \to 0
\end{align*}
uniformly by Theorem \ref{T:unispectrum}.  
Hence we define 
$$H_\infty(x) := c\cdot \langle m_0, P(\chi_{[0, x]}) \rangle = \frac{\int P(\chi_{[0, x]}) \ dm_0}{\int P(1) \ dm_0} $$
and we obtain 
$$H_n \to H_\infty$$
uniformly.
Then, we have 
$$f_n = H_n \circ f \circ H_n^{-1}.$$
Given $x \in I$, let $y \in H_\infty^{-1}(x)$ and define 
$$f_\infty(x) := H_{\infty}(f(y)).$$
Note that, if $[y_1, y_2]$ is an interval where $H_\infty$ is constant, then 
$H_\infty$ is also constant on $[f(y_1), f(y_2)]$, hence 
$f_\infty$ is well-defined.

\begin{lemma}
$H_\infty$ is constant precisely on slow intervals.
\end{lemma}

\begin{proof}
In fact, by eq. \eqref{E:Hn1}, $H_\infty$ is constant on $J$ if and only if  
\begin{equation} \label{E:core}
\lim_{n \to \infty} \frac{\langle m_0,  \mathcal{L}^n (\chi_J) \rangle}{M_n} = 
\lim_{n \to \infty} \frac{\textup{Var}_J (f^n)}{\textup{Var}_{I} (f^n)} = 0,
\end{equation}
where $I  = [0,1]$ is the whole interval. If $J$ is slow, then 
$$\textup{Var}_J (f^n) \leq \ell(f^n, J) \leq 2^{n/2}$$ 
and $\lambda > \sqrt{2}$, hence the limit above is zero. 
If $J$ is fast, then there exists $k$ such that $f^k(J) = K$ the core of $f$, hence 
$$\textup{Var}_J (f^{n+k}) \geq \textup{Var}_{f^k(J)} (f^{n}) = \textup{Var}_{K} (f^{n}) $$
hence, since 
$\frac{\textup{Var}_{K} (f^{n})}{\textup{Var}_{I} (f^{n})}$ is bounded below independently of $n$,
the limit in \eqref{E:core} is positive.
\end{proof}

Now, (3) follows from: 

\begin{lemma} \label{L:f-convergence}
We have 
$$f_n \to f_\infty$$
uniformly. 
\end{lemma}

\begin{proof}
Given the definitions of $H_n$, $H_\infty$ above, the claim follows if we prove 
$$\sup_{x \in I} d(H_n^{-1}(x), H_\infty^{-1}(x)) \to 0,$$
where $d$ is the euclidean distance. Let us see the proof. 

Fix $\epsilon > 0$ so that neither $x = \epsilon$ nor $x = 1 - \epsilon$ lie in a slow interval, and let $K_\epsilon$ be the complement of the union of open slow intervals with diameter $\geq 2 \epsilon$. 
Define 
$$\delta(x, \epsilon) :=  \min \{ H_\infty(x + \epsilon) - H_\infty(x), H_\infty(x) - H_\infty(x -\epsilon)\}.$$
Note that for any $x \in K_\epsilon \cap [\epsilon, 1 - \epsilon]$ we have $\delta(x, \epsilon) > 0$. 
Indeed, if not then $H_\infty(x + \epsilon) = H_\infty(x) = H_\infty(x-\epsilon)$, hence $x$ lies 
in the interior of a slow interval of length at least $2 \epsilon$, which contradicts the definition of 
$K_\epsilon$. 

Since $K_\epsilon \cap [\epsilon, 1 - \epsilon]$ is compact, then 
$$\delta(\epsilon) := \inf_{x \in K_\epsilon \cap [\epsilon, 1 - \epsilon]} \delta(x, \epsilon)  > 0.$$
Since $H_n$ converges to $H_\infty$ uniformly,  let $n_0$ such that 
$$\Vert H_n - H_\infty \Vert_\infty < \frac{\delta(\epsilon)}{2} \qquad \textup{for }n \geq n_0.$$ 
Now, let $y \in I$. Let $x_1, x_2$ be the two endpoints of $H_\infty^{-1}(y)$, with $x_1 \leq x_2$. 

Suppose that $[x_1, x_2] \subseteq [\epsilon, 1 - \epsilon]$. 
Then, 
$$H_n(x_2 + \epsilon) \geq H_\infty(x_2+ \epsilon) - \frac{\delta(\epsilon)}{2} 
\geq H_\infty(x_2) + \frac{\delta(\epsilon)}{2}$$
and similarly 
$$H_n(x_1 - \epsilon) \leq H_\infty(x_1) - \frac{\delta(\epsilon)}{2}$$
hence there exists a point $p \in [x_1 - \epsilon, x_2 + \epsilon]$ with $H_n(p) = H_\infty(x_1) = y$.

Suppose otherwise that $x_1 \leq \epsilon$. Then 
$$H_\infty(2 \epsilon) \geq H_\infty(\epsilon) + \delta(\epsilon)$$
hence 
$$H_n(2 \epsilon) \geq H_\infty(\epsilon) + \frac{\delta(\epsilon)}{2} = y + \frac{\delta}{2}$$
whereas 
$$H_n(0) = H_\infty(0) = 0$$
hence there exists $p \in [0, 2 \epsilon]$ such that $H_n(p) = H_\infty(x_1) = y$.

The case $x_2 \geq 1 - \epsilon$ is symmetric.
\end{proof}

\medskip
\noindent 
Finally, (4) follows from (3) and the following conditional convergence result of Milnor.

\begin{theorem}[\cite{Mil-slides}, Slide 22] \label{T:Mil-conditional}
If $(f_n)$ converges uniformly to $f_\infty$, then $(g_n)$ and $(h_n)$ also converge uniformly. 
Moreover, the limit maps $f_\infty, g_\infty, h_\infty$ commute with each other.
\end{theorem}

Finally, by passing to the limit in \eqref{E:defHn}, the map $H_\infty$ semiconjugates $f_0$ to $f_\infty$. 
Moreover, as in Lemma \ref{L:semiconj}, since $m_\infty$ is an eigenvector of $P$ with eigenvalue $\lambda = e^{h_{top}(f)}$, the map $f_\infty$ is piecewise linear of slope $\lambda$. 
\end{proof}

In fact, the previous discussion applies verbatim to any piecewise monotone map which is mixing for the measure of maximal entropy; 
in particular, to topologically mixing maps. Hence, Theorem \ref{T:unimodal-conv} holds for them as well.  

\section{Counterexamples} \label{S:counter}

\subsection{Higher degree case}
We now show that the iterative procedure above described does not always converge (even if we assume positive entropy). 
Let $a, b$ be two positive constants, with $a + b = 1$. Let $f$ be such that 
$f(0) = 1$, $f(a) = a$, or $f(1) = 0$. Moreover, $f$ maps the interval $[0, a]$ three times onto $[a, 1]$, 
and maps $[a, 1]$ three times onto $[0, a]$. 

\begin{figure}[h!]
\includegraphics[width = 0.7 \textwidth]{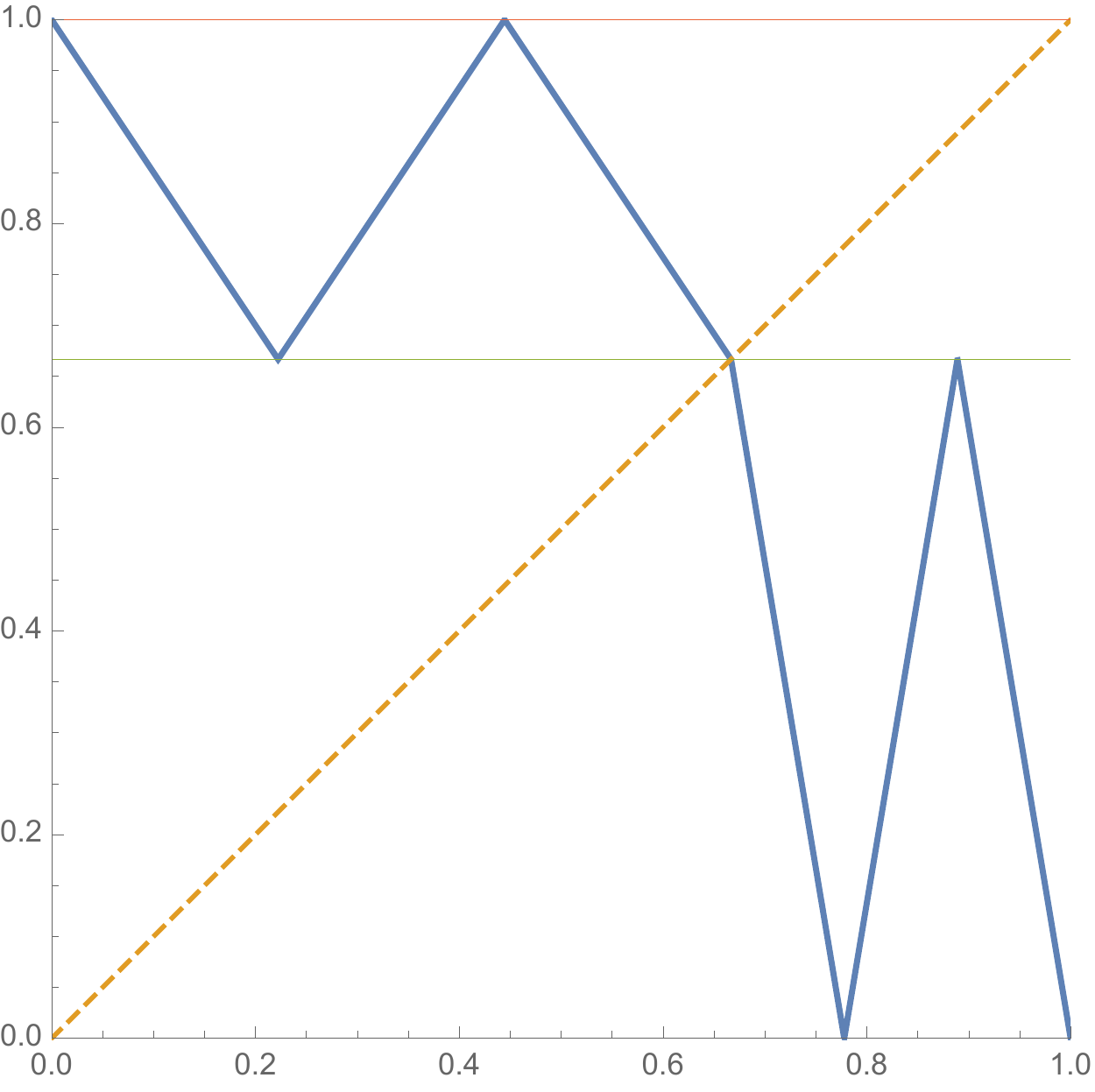}
\caption{A counterexample to convergence.}
\end{figure}

In fact, one can observe by the definition that 
$$(f^\star)^N([0, a]) = 3 (f^\star)^{N-1}([a, 1]) = 9 (f^\star)^{N-2}([0, a])$$
which, since we know $m_0([0, a]) = a$ and $m_0([a, 1]) = b$, implies that 
$$(f^\star)^N([0, a]) = \left\{ \begin{array}{ll}
3^N a & \textup{if }N \textup{ is even} \\
3^N b & \textup{if }N \textup{ is odd.} 
\end{array}\right.$$
Hence, 
$$m_N([0, a]) = \left\{ \begin{array}{ll}
a & \textup{if }N \textup{ is even} \\
b & \textup{if }N \textup{ is odd} 
\end{array}\right.$$
thus $m_N$, and hence $H_N$, does not converge as we chose $a \neq b$. Indeed, 
$H_N(x) = x$ if $N$ is even, and 
$$H_N(x) = \left\{ \begin{array}{ll} \frac{b}{a} x & \textup{if }0 \leq x \leq a \\ 
b + \frac{a}{b} (x-a) &  \textup{if }a \leq x \leq 1 \end{array} \right.$$
if $N$ is odd. Similarly, $f_N = f$ for $N$ even, while $f_N$ maps $[0, b]$ to itself if $N$ odd. 
Hence the sequence $(f_n)$ does not converge either.

\subsection{Unimodal case}
Let us note that the condition $h_{top}(f) > \frac{1}{2} \log 2$ in Theorem \ref{T:unimodal-conv} is needed to guarantee convergence. 
To see why, let us assume that $f$ is the basilica tuned with the airplane, i.e. where the large scale dynamics is 
given by the basilica, and the small scale dynamics is the airplane. 
This is given by the external angle $\theta = \frac{26}{63} = .\overline{011010}$.
Then all entropy is concentrated in the small Julia set, which moves with period $2$. Thus, 
it is not hard to find a metric (e.g. one concentrated on one of the two copies of the Julia set) which 
does not converge, as the odd and even iterates are different.
In general one has: 

\begin{theorem}
For any $0 < h \leq \frac{1}{2} \log(2)$, there exists a unimodal map of topological entropy $h$ 
such that the sequences $(H_n)$ and $(f_n)$ do not converge uniformly.  
\end{theorem}

\begin{proof}
In fact, every real quadratic polynomial with entropy $h \leq \frac{1}{2} \log 2$ is the tuning of the basilica with some other polynomial. 
Thus, there exist two intervals $I_0, I_1$ with $I_0 \cup I_1 = K_f$ and intersecting only at a fixed point for $f$, such that  $f(I_0) = I_1$, $f(I_1) = I_0$. Let us label them so that $I_0$ contains the turning point. Then, one considers the metrics 
$$\mu_0 := \frac{m_0 \vert_{I_0}}{m_0(I_0)} \qquad \qquad \mu_1 := \frac{m_0 \vert_{I_1}}{m_0(I_1)}$$
which are the normalized restrictions of the Lebegue measure to $I_0, I_1$, respectively. 
Since the orbits of points under $f$ alternate between $I_0$ and $I_1$, the metric $P^n(\mu_i)$ is supported on $I_i$ for $n$ even, and supported on $I_{i+1}$ for $n$ odd 
(where the index $i$ is taken modulo $2$).

Now, for any $0 < \alpha < 1$, let us consider the metric
$$m^{(\alpha)} := \alpha \mu_0 + (1- \alpha) \mu_1.$$
For each $\alpha$, let us define $h_\alpha(x) := m^{(\alpha)}([0, x])$ and consider 
$f_\alpha := h_\alpha \circ f \circ h_\alpha^{-1}$.
Since $f_\alpha$ is topologically conjugate to $f$, it has the same topological entropy. 
We shall show that the tower algorithm can converge for at most one value of $\alpha$, which implies the claim.

Let $a_n := (f^{*n} \mu_0)(I_0)$.  Then, using the fact that $f$ maps $I_i$ to $I_{i+1}$ for $i = 0, 1$, we compute 
$$a_n^{(\alpha)} := (f^{*n} m^{(\alpha)})(I_0) =  \left\{ \begin{array}{ll} 
\alpha a_n & \textup{if }n \equiv 0 \mod 2 \\
(1-\alpha) a_n & \textup{if }n \equiv 1 \mod 2 .
\end{array} \right.$$
Moreover, since $f$ maps $I_1$ homeomorphically onto $I_0$, 
$$b_n^{(\alpha)} := (f^{*n} m^{(\alpha)} )(I_1) = (f^{*n-1} m^{(\alpha)} )(I_0) = a_{n-1}^{(\alpha)}.$$
Hence, if we set $m_n^{(\alpha)} := P^n(m^{(\alpha)})$, we compute 
$$m^{(\alpha)}_n(I_0) = \frac{a_n^{(\alpha)}}{a_n^{(\alpha)} + a_{n-1}^{(\alpha)}} = 
\left\{ \begin{array}{ll} \frac{\alpha a_n}{ \alpha a_n + (1-\alpha) a_{n-1}}  = \frac{1}{1 + \frac{1-\alpha}{\alpha} \frac{a_{n-1}}{a_n}}  & \textup{if }n \equiv 0 \mod 2 \\ 
\frac{(1-\alpha) a_n}{(1-\alpha) a_n + \alpha a_{n-1}}  =  \frac{1}{1 + \frac{\alpha}{1-\alpha} \frac{a_{n-1}}{a_n}} & \textup{if }n \equiv 1 \mod 2
\end{array}\right.$$
hence 
$\lim_{n \to \infty} m^{(\alpha)}_n(I_0)$ exists only if $\alpha = \frac{1}{2}$. 

Now, let us consider the tower algorithm for $f_\alpha$, and denote as $(f^{(\alpha)}_n)$ and $(H_n^{(\alpha)})$
the sequences produced by the algorithm starting with $f_\alpha$. 

Since $(f_\alpha)^n \circ h_\alpha = h_\alpha \circ f^n$, we obtain, as the pullback is contravariant
and using the definition $m^{(\alpha)} = h_\alpha^* (m_0)$, 
$$h_\alpha^*  f_\alpha^{*n} (m_0) = f^{*n}  h_\alpha^* (m_0) = f^{*n} m^{(\alpha)}. $$ 
Hence, a computation shows for any $x \in [0, 1]$, 
$$H_n^{(\alpha)}(x) = \frac{(f^{* n}_\alpha m_0)([0, x])}{(f^{* n}_\alpha m_0)([0, 1])} =  \frac{(f^{* n} m^{(\alpha)})([0, h_\alpha^{-1}(x)])}{(f^{* n}m^{(\alpha)})([0, 1])}  = m_n^{(\alpha)}([0, h_\alpha^{-1}(x)]).$$
Now, if we denote by $p$ the largest endpoint of $I_0$, and assume $(H_n^{(\alpha)})$ converges, we have 
$$\lim_{n \to \infty} H_n^{(\alpha)}(h_\alpha(p)) = \lim_{n \to \infty} m^{(\alpha)}_n([0, p])$$
exists, which implies $\alpha = \frac{1}{2}$. 

Finally, let us note that $p_n^{(\alpha)} := H_n^{(\alpha)}(h_\alpha(p))$ is the only fixed point of $f_n^{(\alpha)}$. 
Hence, if $(f_n^{(\alpha)})$ converges uniformly, then also $(p_n^{(\alpha)})$ should converge as $n \to \infty$, 
which as before implies $\alpha = \frac{1}{2}$. 

Since we can set $\alpha$ to any arbitrary value, the claim is proven. 
\end{proof}

\section*{Appendix. Relation with Thurston's algorithm} \label{appendix}

In this section, we compare Thurston's iteration on Teichm\"uller space with the tower algorithm. 

In the classical formulation of Thurston's algorithm \cite{DH}, the iteration to produce a rational map 
with given combinatorial data works as follows.
Let $f : S^2 \to S^2$ be a postcritically finite, orientation-preserving, branched covering of the sphere, and let $P \subseteq S^2$ be a finite set which contains the critical values of $f$ and such that $f(P) = P$. 
We denote as $\mathcal{T}_f$ the Teichm\"uller space of $(S^2, P)$. 
Elements of $\mathcal{T}_f$ are represented by diffeomorphisms $\varphi : (S^2, P) \to \mathbb{P}^1$, up to suitable equivalence: more precisely, $\varphi_1 \sim \varphi_2$ if there exists a conformal isomorphism $h : \mathbb{P}^1 \to \mathbb{P}^1$ such that $h \circ \varphi_1 \vert_{P}= \varphi_2 \vert_{P}$ and moreover $h \circ \varphi_1$ is isotopic to $\varphi_2$ rel $P$.
We denote as $[\varphi]$ the class of $\varphi$ as a point in Teichm\"uller space.
Let $\sigma_f : \mathcal{T}_f \to \mathcal{T}_f$ be the pullback map. 
The main iterative step is described as follows.

\begin{proposition}[\cite{DH}, Proposition 2.2]
If $\tau \in \mathcal{T}_f$ is represented by $\varphi : (S^2, P) \to \mathbb{P}^1$, then $\tau' := \sigma_f(\tau)$ is represented by $\varphi': (S^2, P) \to \mathbb{P}^1$ so that 
$$F := \varphi \circ f \circ (\varphi')^{-1} : \mathbb{P}^1 \to \mathbb{P}^1$$
is a rational map. 
\end{proposition}

By applying the proposition repeatedly, one obtains a sequence $(\varphi_n)$ of homeomorphisms $(S^2, P) \to \mathbb{P}^1$, which fit in the following tower: 
$$\xymatrix{
S^2 \ar[r]^{\varphi_{n+1}}  \ar[d]^f & \mathbb{P}^1 \ar[d]^{F_n} \\
S^2 \ar[r]^{\varphi_{n}}  \ar@{..>}[d] & \mathbb{P}^1 \ar@{..>}[d] \\
S^2 \ar[r]^{\varphi_2}  \ar[d]^f & \mathbb{P}^1 \ar[d]^{F_1} \\
S^2 \ar[r]^{\varphi_1}  \ar[d]^f & \mathbb{P}^1 \ar[d]^{F_0} \\
S^2 \ar[r]^{\varphi_0}  & \mathbb{P}^1  
}$$
where each $F_n$ is a rational map, and the relation 
\begin{equation} \label{E:th1}
F_{n} \circ \varphi_{n+1} = \varphi_n \circ f
\end{equation}
holds for any $n \geq 0$. 

\bigskip
Another iterative scheme is introduced in \cite{Mil}; let us compare the two. 
To convert to the terminology of \cite{Mil}, let us fix a set $Q$ of three points, with $Q \subseteq P$. 
We can identify $S^2$ with $\mathbb{P}^1$ and start our iteration with $\varphi_0 = \textup{id}$. 
Moreover, by composing with M\"obius maps, we can make sure that the restriction of $\varphi_n$
to $Q$ is the identity for each $n$.

Then we set the following: 
\begin{align}
f_n & := \varphi_n \circ f \circ \varphi_n^{-1} \label{E:1} \\
h_n & := \varphi_{n+1} \circ \varphi_n^{-1}  \label{E:2} \\
g_n & := F_n.
\end{align}
Then $f_n$ is topologically conjugate to $f$, $h_n$ is a homeomorphism and $g_n$ is a rational map. 
Finally, by construction our $h_n$'s are normalized so that the restriction of each $h_n$ to $Q$ is the identity.

\begin{lemma}
Under the above definitions, we have the identities
$$g_n \circ h_n = f_n$$
$$h_n \circ g_n = f_{n+1}$$
for any $n \geq 0$. 
\end{lemma}

\begin{proof}
Using \eqref{E:1}, \eqref{E:2}, and \eqref{E:th1}, one checks
\begin{align*}
g_n \circ h_n & = (\varphi_n \circ f \circ \varphi_{n+1}^{-1})  \circ  (\varphi_{n+1} \circ \varphi_n^{-1}) \\
 & = \varphi_n \circ f \circ \varphi_{n}^{-1} \\
 & = f_{n}.
\end{align*}
Similarly, 
\begin{align*}
h_n \circ g_n & = (\varphi_{n+1} \circ \varphi_n^{-1}) \circ (\varphi_n \circ f \circ \varphi_{n+1}^{-1}) \\
 & = \varphi_{n+1} \circ f \circ \varphi_{n+1}^{-1} \\
 & = f_{n+1}
\end{align*}
as needed.
\end{proof}

Thus, this iteration scheme fits into the tower
$$\xymatrix{
S^2 \ar[r]^{f_{n+1}}  \ar[dr]^{g_n} & S^2  \\
S^2 \ar[r]^{f_n} \ar[u]^{h_n}   & S^2 \ar[u]_{h_n}  \\
S^2 \ar[r]^{f_2}  \ar[dr]^{g_1} \ar@{..>}[u] & S^2 \ar@{..>}[u] \\
S^2 \ar[r]^{f_1}  \ar[dr]^{g_0} \ar[u]^{h_1} & S^2  \ar[u]_{h_1}\\
S^2 \ar[r]^{f_0} \ar[u]^{h_0} & S^2 \ar[u]_{h_0} 
}$$

which is precisely the iterative process discussed in \cite{Mil}. 
In particular, as a corollary of Thurston's theorem, one gets the following convergence result, 
as mentioned in \cite{Mil-slides}.

An orientation-preserving branched covering of the sphere $f$ is called a \emph{topological real polynomial} 
if $f^{-1}(\infty) = \infty$ and $f$ preserves the real axis, i.e. $f(\mathbb{R}) \subseteq \mathbb{R}$ (recall that we have fixed an identification between $S^2$ and $\mathbb{P}^1$). 
A topological real polynomial of degree $d$ is in \emph{normal form} if $f(0) = 0$, $f(1) \in \{0, 1\}$ and $f$ has $d-1$ distinct real critical points, all of which lie in the interior of $I = [0, 1]$.

\begin{theorem}[\cite{Mil-slides}]
Suppose $f$ is a postcritically finite topological real polynomial in normal form with at least $4$ postcritical points, and has no obstruction. Then the sequence $(g_n)$ converges uniformly to a real polynomial map $g_\infty$.
\end{theorem}

\begin{proof}
Let us set $Q = \{0, 1, \infty\}$ and $P = P_f \cup Q$, where $P_f$ is the postcritical set of $f$.
By Thurston's theorem \cite{DH}, there exists a homeomorphism $\varphi_\infty : S^2 \to \mathbb{P}^1$ so that 
$[\varphi_n] \to [\varphi_\infty]$ in the Teichm\"uller metric. This also yields a rational map $F_\infty$ 
which is (Thurston)-equivalent to $f$, meaning that $F_\infty \circ \varphi_\infty\vert_{P} = \varphi_\infty \circ f\vert_{P}$ and $F_\infty \circ \varphi_\infty \sim \varphi_\infty \circ f$ 
up to isotopy rel $P$.

This means that there exists for each $n$ a quasi-conformal map 
$\iota_n : \mathbb{P}^1 \to \mathbb{P}^1$ with 
$\iota_n \circ \varphi_n \sim \varphi_\infty$ $\textup{rel }P$ and 
$\iota_n \circ \varphi_n\vert_{P} = \varphi_\infty \vert_{P}$
so that the quasi-conformality constant satisfies $K(\iota_n) \to 1$. 
By our normalization, we also have $\iota_n \vert_Q = \textup{id}$.
This implies that for any $p \in P$ we have 
$$\varphi_n(p) \to \varphi_\infty(p)$$
as $n \to \infty$. 
In particular, the critical values of $F_n$ converge to the critical values of $F_\infty$. 
Let $\mathcal{F}_d$ be the set of real polynomials of degree $d$ in normal form.
Since the map 
$$\mathcal{F}_d \ni f \mapsto (v_1, \dots, v_{d-1}) \in \mathbb{R}^{d-1}$$ 
which assigns to a polynomial in normal form its critical values is a homeomorphism onto its image (\cite[Lemmas 3.1, 3.2]{BMS}, see also \cite[Appendix A]{MiTr}, \cite{Ch}), this implies that $F_n$ converges uniformly to $F_\infty$. (Recall that $g_n = F_n$ by definition).
\end{proof}


\begin{thebibliography}{99}

\bibitem[AM]{Avila-Moreira}
A. Avila, C. G. Moreira, \emph{Statistical properties of unimodal maps: the quadratic family}, 
Ann. Math. 161 (2005), 831--881.

\bibitem[Ba1]{Baladi-book}
V. Baladi, \emph{Positive transfer operators and decay of correlations}, 
World Scientific, Singapore, 2000.

\bibitem[Ba2]{Baladi-notes}
V. Baladi, \emph{Dynamical zeta functions},
available at \url{arXiv:1602.05873}. 

\bibitem[BK]{Baladi-Keller}
V. Baladi, G. Keller, 
\emph{Zeta functions and transfer operators for piecewise monotone transformations},
Comm. Math. Phys. 127 (1990), 459--477.

\bibitem[BR]{Baladi-Ruelle}
V. Baladi, D. Ruelle,  \emph{Sharp determinants}, 
Invent. Math. 123 (1996), 553--574.

\bibitem[BL]{Blokh-Lyubich}
A. M. Blokh, M. Yu. Lyubich, \emph{Measurable dynamics of S-unimodal maps of the interval}, 
Ann. Scient. \'Ec. Norm. Sup., 24 (1991), 545--573.

\bibitem[BMS]{BMS}
A. Bonifant, J. Milnor and S. Sutherland, 
\emph{The W. Thurston Algorithm Applied to Real Polynomial Maps}, 
Conform. Geom. Dyn. 25 (2021), 179--199. 

\bibitem[Bo]{Bo}
R. Bowen, \emph{Bernoulli maps of the interval}, 
Israel J. Math. 28 (1922), nos. 1-2, 161--168. 

\bibitem[BH]{BH}
S. Boyd and C. Henriksen, \emph{The Medusa algorithm for polynomial matings}, 
Conform. Geom. Dyn. 16 (2012), 161--183.

\bibitem[Ch]{Ch}
D. Chandler, \emph{Extrema of a Polynomial}, 
Amer. Math. Monthly 64 (1957), no. 9, 679--680.

\bibitem[DH]{DH}
A. Douady and J. Hubbard, \emph{A proof of Thurston's topological characterization of rational functions}, 
Acta Math. 17 (1993), 263--297.

\bibitem[Ho]{Ho}
F. Hofbauer, 
\emph{On intrinsic ergodicity of piecewise monotonic transformations with positive entropy}, 
Israel J. Math. 34 (1979), no. 3, 213--237.

\bibitem[HK1]{HK2}
F. Hofbauer, G. Keller, 
\emph{Ergodic Properties of Invariant Measures for Piecewise Monotonic Transformations}, 
Math. Z. 180 (1982), 119--140.

\bibitem[HK2]{HK}
F. Hofbauer, G. Keller, 
\emph{Zeta-functions and transfer-operators for piecewise linear transformations}, 
J. Reine Angew. Math. 352 (1984), 100--113.

\bibitem[Ke]{Ke}
G. Keller, 
\emph{On the Rate of Convergence to Equilibrium in One-Dimensional Systems}, 
Commun. Math. Phys. 96 (1984), 181--193.

\bibitem[Mi1]{Mil} 
J. Milnor, \emph{Thurston's algorithm without critical finiteness}, in \emph{Linear and
Complex Analysis Problem Book 3, Part 2}, Havin and Nikolskii editors, Lecture
Notes in Math no. 1474, pp. 434-436, Springer, 1994.

\bibitem[Mi2]{Mil-slides}
J. Milnor, \emph{Remarks on Piecewise Monotone Maps}, 2015, 
available at \texttt{https://www.math.stonybrook.edu/~jack/BREMEN/pm-print.pdf}.

\bibitem[MiTr]{MiTr}
J. Milnor and C. Tresser,
\emph{On Entropy and Monotonicity for Real Cubic Maps}, 
Comm. Math. Phys. 209 (2000), 123--178. With an Appendix by A. Douady and P. Sentenac.

\bibitem[MT]{MT}
J. Milnor and W. Thurston, 
\emph{On iterated maps of the interval}, in \emph{Dynamical Systems}, eds. J.C. Alexander, 
Lecture Notes in Mathematics, vol 1342. Springer, Berlin, Heidelberg, 1988.

\bibitem[MS]{MS}
M. Misiurewicz and W. Szlenk, 
\emph{Entropy of piecewise monotone mappings}, 
Studia Math. 67 (1980), 45--63.

\bibitem[Pr]{Preston}
C. Preston, 
\emph{What you need to know to knead}, 
Adv. Math. {\bf 78} (1989), 192--252.

\bibitem[Ra]{Ra}
P. Raith, \emph{Continuity of the Measure of Maximal Entropy for Unimodal Maps on the Interval}, 
Qualit. Theory Dynam. Systems 4 (2003), 67--76.

\bibitem[Ru]{Rugh}
H. H. Rugh, 
\emph{The Milnor-Thurston determinant and the Ruelle transfer operator}, 
Comm. Math. Phys. 342 (2016), 603--614. 

\bibitem[Ry]{Ry}
M. Rychlik, \emph{Bounded variation and invariant measures}, 
Studia Math. 76 (1983), 69--80.

\bibitem[vS]{vS}
S. van Strien, \emph{Smooth Dynamics on the Interval (with an emphasis on quadratic-like maps)}, 
in \emph{New Directions in Dynamical Systems}, eds. 
T. Bedford, H. Swift, Cambridge University Press, 1988, 57--119.

\bibitem[Ti]{Ti}
G. Tiozzo, \emph{The local H\"older exponent for the entropy of real unimodal maps}, 
Science China Mathematics 61 (2018), no. 12, 2299--2310.

\bibitem[Wi]{Wi}
M. Wilkerson, \emph{Thurston's algorithm and rational maps from quadratic polynomial matings},
Discrete Contin. Dyn. Syst. Ser. S 12 (2019), no. 8, 2403--2433.

\end{thebibliography}
\end{document}